\documentclass[12pt,a4paper,psamsfonts]{amsart}
\usepackage{amssymb,amscd,amsxtra,calc}
\usepackage{cmmib57}

\theoremstyle{plain}
    \newtheorem{thm}{Theorem}[section]

    \newtheorem{appl}[thm]{Application}
    
    \newtheorem{lemma}[thm]{Lemma}

    \newtheorem{theorem}[thm]{Theorem}

\theoremstyle{definition}

    \newtheorem{remark}[thm]{Remark}
\theoremstyle{remark}

\theoremstyle{question}

\newcommand{\C}{\mathbb{C}}

\newcommand{\GG}{\mathbb{G}}

\newcommand{\PP}{\mathbb{P}}
\newcommand{\Q}{\mathbb{Q}}

\newcommand{\cit}{\mathbb{C}}

\newcommand{\pit}{\mathbb{P}}

\newcommand{\alb}{\operatorname{alb}}
\newcommand{\Aut}{\operatorname{Aut}}

\newcommand{\lin}{\operatorname{lin}}

\newcommand{\NS}{\operatorname{NS}}

\newcommand{\Alb}{\operatorname{Alb}}

\begin{document}


\title[A characterization of compact complex tori]
{A characterization of compact complex tori via automorphism groups}

\author{Baohua Fu}
\address
{ \textsc{Institute of Mathematics, AMSS, Chinese Academy of
Sciences,} \endgraf \textsc{55 ZhongGuanCun East Road, Beijing,
100190, P. R. China}} \email{bhfu@math.ac.cn}

\author{De-Qi Zhang}
\address
{
\textsc{Department of Mathematics, National University of Singapore,} \endgraf
\textsc{
10 Lower Kent Ridge Road,
Singapore 119076
}}
\email{matzdq@nus.edu.sg}

\begin{abstract}
 We  show that a compact K\"ahler manifold
$X$ is a complex torus if both the continuous part and discrete part of some automorphism group $G$ of $X$ are infinite groups, unless
$X$ is bimeromorphic to a non-trivial $G$-equivariant fibration. Some applications to dynamics are given.
\end{abstract}



\maketitle

\section{Introduction}\label{Intro}

We work over the field $\C$ of complex numbers.
Let $X$ be a compact K\"ahler manifold. Denote by $\Aut(X)$ the automorphism group of $X$ and by $\Aut_0(X)$ the identity connected component of $\Aut(X)$. By \cite{Fu}, $\Aut_0(X)$ has a natural meromorphic group structure.
  Further there exists a unique meromorphic subgroup, say $L(X)$, of
$\Aut_0(X)$, which is meromorphically isomorphic to a linear algebraic group and such
that the quotient $\Aut_0(X)/L(X)$ is a complex torus.
In the following, by a subgroup of $\Aut_0(X)$ we always mean a meromorphic subgroup and by a linear algebraic subgroup
of $\Aut_0(X)$ we mean a Zariski closed meromorphic subgroup contained in $L(X)$.

For a subgroup $G \le \Aut(X)$, the
pair $(X, G)$ is called {\em strongly primitive} if for
every finite-index subgroup $G_1$ of $G$,
$X$ is not bimeromorphic to a non-trivial $G_1$-equivariant fibration,
i.e., there does not exist any compact K\"ahler manifold
$X'$ bimeromorphic to $X$, such that $X'$ admits  a
$G_1$-equivariant holomorphic map $X' \to Y$ with $0 < \dim Y < \dim X$  and  $G_1 \le \Aut(X')$.
From the dynamical point of view, these manifolds are essential.
Our main result Theorem \ref{ThA} says that for these manifolds, unless it is  a complex torus,
there is no interesting dynamics if its automorphism group has positive dimension.

\begin{theorem}\label{ThA}
Let $X$ be a compact K\"ahler manifold and $G \le \Aut(X)$ a
subgroup of automorphisms. Assume the following three conditions.
\begin{itemize}
\item[(1)] $G_0 := G \cap \Aut_0(X)$ is infinite.
\item[(2)]  $|G : G_0| = \infty $.
\item[(3)] The pair $(X, G)$ is strongly primitive.
\end{itemize}
Then $X$ is a complex torus.
\end{theorem}

As a key step towards Theorem \ref{ThA}, we prove the following result.
A proof for Theorem \ref{ThB}(2) is long overdue
(and we do it geometrically via \ref{ThB}(1)), but the authors could not find
it in any literature, even after consulting many experts across the continents.

\begin{theorem}\label{ThB}
Let $X$ be a compact K\"ahler manifold and $G_0 \subset \Aut_0(X)$ a linear algebraic subgroup.
Assume that $G_0$ acts on $X$ with a Zariski open dense orbit.
Then we have:
\begin{itemize}
\item[(1)]
$X$ is projective; the anti canonical divisor $-K_X$ is big, i.e.
$\kappa(X, -K_X) = \dim X$.
\item[(2)]
$\Aut(X)/\Aut_0(X)$ is finite.
\end{itemize}
\end{theorem}

\begin{remark}\label{rThm}
(i) The condition (2) in Theorem \ref{ThA} is satisfied if  $G$ acts on  $H^2(X, \C)$
as an infinite group, or if $G$ has an element of positive entropy (cf.~\S \ref{setup1.1}
for the definition).

(ii)  Theorem \ref{ThB} implies that when $\dim X \ge 3$
the case(4) in \cite[Theorem 1.2]{Z-Tits} does not occur, hence it can be removed from the statement.

(iii) Theorem \ref{ThA} generalizes \cite[Theorem 1.2]{CWZ}, where it is proven for $\dim X = 3$ and under an additional assumption.

(iv) A similar problem
for endomorphisms of
homogeneous varieties has been studied by S. Cantat in \cite{Ca2}.
\end{remark}


Two applications are given. The first one
generalizes  the following result due to Harbourne  \cite[Corollary (1.4)]{Hb} to higher dimension:
if $X$ is a smooth projective rational surface with $\Aut_0(X) \ne (1),$
then $\Aut(X)/\Aut_0(X)$ is finite. To state it, recall (\cite[Theorem 4.1]{Fu})
that for any connected subgroup $H \le \Aut(X)$, there
exist a quotient space $X/H$ and an $H$-equivariant dominant meromoprhic map $X \dasharrow X/H$,
which satisfies certain universal property.

\begin{appl}\label{Cor1}
Let $X$ be a compact K\"ahler manifold with irregularity $q(X) = 0$.
Suppose that the quotient space $X/\Aut_0(X)$
has dimension $\le 1$. Then $X$ is projective and $\Aut(X)/\Aut_0(X)$ is finite.
\end{appl}

The second application essentially says that when we study dynamics of a compact K\"ahler manifold $X$,
we may assume that $\Aut_0(X)_{\lin} = (1)$, where $\Aut_0(X)_{\lin}$ is the largest connected linear
algebraic subgroup of $\Aut_0(X)$.

\begin{appl}\label{ThC}
Let $X$ be a smooth projective variety and $G_0 \lhd G \le \Aut(X)$.
Suppose that $G_0$ is a connected linear closed subgroup of $\Aut_0(X)$.
Let $Y$ be a $G$-equivariant resolution of the quotient space
$X/G_0$ and replace $X$ by a $G$-equivariant resolution
so that the natural map $\pi : X \to Y$
is holomorphic.
Then for any $g \in G$,
we have the equality of the first dynamical degrees:
$$d_1(g_{| X}) = d_1(g_{| Y}),$$
where $d_1(g_{| X}) := \max\{|\lambda| \, ; \, \lambda \,\, \text{is
an eigenvalue of} \,\, g^* \, | \, H^{1,1}(X) \}$.

In particular, $G_{| X}$ is of null entropy if and only if so is $G_{| Y}$
{\rm (cf.~\ref{setup1.1}).}
\end{appl}

If $q(X)=0$, then $\Aut_0(X)$ (hence $G_0$) is always a linear algebraic group. On the other hand,
if $G_0$ is not linear, then Application \ref{ThC} does not hold (cf.~Section \ref{Rmk}).

\par \vskip 0.5pc \noindent
{\bf Acknowledgement.} We would like to thank Michel Brion for suggesting a proof of
Theorem \ref{ThB} for spherical varieties, Michel Brion and Alan Huckleberry for very patiently explaining
to the second author
basic results on almost homogeneous varieties, the referee, Michel Brion, Jun-Muk Hwang and Akira Fujiki for clarifying Serre's example of non-algebraic action
of complex torus $\cit^* \times \cit^*$ on a $\pit^1$-bundle over an elliptic curve, and Serge Cantat for encouraging us to strengthen
\cite[Theorem 1.2]{CWZ} to the current Theorem \ref{ThA} (perhaps with some extra conditions)
and observing that the Albanese map in Theorem \ref{ThA} is indeed an isomorphism instead of
our original assertion of being bimeromorphic.
B. Fu is supported by
NSFC 11031008.  D.-Q. Zhang is supported by an ARF of NUS.

 \renewcommand{\thethm}
    {\arabic{section}.\arabic{subsection}.\arabic{thm}}

\section{Proof of Theorems}

\subsection{}
For a compact K\"ahler manifold $X$, denote by $\NS_\mathbb{R}(X)$ its Neron-Severi group. For an
element $[A] \in \NS_\mathbb{R}(X)$, let $\Aut_{[A]}(X) := \{\sigma \in \Aut(X) \, | \, \sigma^*[A] = [A] \,\, \text{in} \,\, \NS_\mathbb{R}(X) \}$.

\begin{lemma}\label{trans} {\rm (cf.~\cite[ChII, Propositions 3.1 and 6.1]{Hs})}
Let $X$ be a projective variety and $U$ an affine open subset of $X$. Then
$D= X \setminus U$ is of pure codimension $1$ and further, when $D$ is $\Q$-Cartier, it is a big divisor,
i.e. the Iitaka $D$-dimension $\kappa(X, D) = \dim X$.
\end{lemma}

\begin{lemma} \label{big}
Let $X$ be a compact K\"ahler manifold and $B$ a big Cartier divisor. Then $X$ is projective.
Let $G \le \Aut(X)$ be a subgroup such that $g^*B \sim B$ for every $g \in G$.
Then $|G : G \cap \Aut_0(X)| < \infty$.
\end{lemma}

\begin{proof}
The existence of a big divisor on $X$ implies that $X$ is Moishezon,
so $X$ is projective since it is also K\"ahler (\cite{Mo}).
Replacing $B$ by a multiple, we may assume that the complete
linear system $|B|$ gives rise to a birational map
$\Phi_{|B|} : X \dasharrow Y$.
Take a $G$-equivariant blowup $\pi : X' \to X$ such that
$|\pi^*B| = |M| + F$ where $|M|$ is base point free and hence nef
and big.  Then
$\Aut_{[M]}(X')$  is a finite
extension of $\Aut_0(X')$ (cf.~\cite[Lemma 2.23]{JDG}, \cite[Proposition 2.2]{Li}).
By the assumption,  $G \le \Aut_{[M]}(X')$.
Set $G_0 := G \cap \Aut_0(X')$. Then $|G : G_0| \le |\Aut_{[M]}(X') : \Aut_0(X')| < \infty$.
Take an ample divisor $A$ on $X$. Then ${G_0}_{| X'}$ fixes the class $[\pi^*A]$
and hence ${G_0}_{ | X} \le \Aut_{[A]}(X)$. As $\Aut_{[A]}(X)$ is a finite extension of
 $\Aut_0(X)$, ${G_0}_{| X}$ is a finite extension of $G_0 \cap \Aut_0(X)$.
Now the lemma follows.
\end{proof}

\begin{lemma} \label{affine}
Let $X$ be a compact K\"ahler manifold and $G_0 \subset \Aut_0(X)$ a linear algebraic subgroup.
Assume that $G_0$  acts on $X$ with a Zariski open dense orbit.
Then $X$ is projective.
Assume furthermore that the open $G_0$-orbit $U$ is isomorphic to $G_0/\Gamma$, for some finite group $\Gamma$.
Then $-K_X$ is big.
\end{lemma}

\begin{proof}
By a classical result of Chevalley, a connected linear algebraic group is a rational variety. By \cite{Fu}, $G_0$ has a compactification $G_0^*$ such that the map $G_0 \times X \to X$ extends to a meromorphic map
$G_0^* \times X \dasharrow X$.
Hence if $G_0$  is a linear algebraic group and has a Zariski dense orbit in $X$,
then $X$ is meromorphically dominated by a rational variety $G_0^*$ and is unirational.  Hence $X$ is Moishezon and also K\"ahler. Thus $X$ is projective.

Let $f : X \to Y$ be a $G_0$-equivariant compactification of the quotient map $G_0 \to G_0/\Gamma$.
By the ramification divisor formula $K_X = f^*K_Y + R_f$ with $R_f$ effective,
to say $-K_Y$ is big, it suffices to say the same for $-f^*K_Y$ or $-K_X$.
So we may assume that $\Gamma = (1)$.

Let $D = \sum_i D_i$ be the irreducible decomposition of
$D := X \setminus U$, which is $G_0$-stable and of pure codimension one since
$U$ is affine and by Lemma \ref{trans}. Furthermore, Lemma \ref{trans}  implies that $D$ is big.
By \cite[Theorem 2.7]{HT}, we have $-K_X = \sum_i a_i D_i$ for some integers $a_i \geq 1$.
Thus $\kappa(X, -K_X) = \kappa(X, D) = \dim X$.
\end{proof}

\subsection{Proof of Theorem \ref{ThB}}

The assertion (2) follows from (1) and Lemma \ref{big} since every automorphism of $X$ preserves
the divisor class $[-K_X]$. We now prove the assertion (1).
Replacing $G_0$ by its connected component, we may assume that $G_0$ is connected.
Let $U = G_0/H$ be the open $G_0$-orbit in $X$, where $H = (G_0)_{x_0}$ is the stabilizer subgroup of $G_0$
at a point $x_0 \in U$.
First we show that $-K_X$ is effective.
Let $\mathfrak{g}$  be the Lie algebra of $G_0$. As $X$ is almost homogeneous, we can take $n = \dim X$ elements
$v_1, \cdots, v_n$ in $\mathfrak{g}$  such that $\sigma =
\tilde{v_1} \wedge \cdots \wedge \tilde{v_n}$ is not identically
zero on $X$, where $\tilde{v_i}$
is the vector field corresponding to $v_i$ via the isomorphism $\mathfrak{g} \simeq H^0(X, T_X)$. Then $\sigma$ gives a
non-zero section of $-K_X$. Hence $-K_X$ is effective.

Let $H_0$ be the identity connected component of $H$ and
$N(H_0)$ its normalizer in $G_0$. We consider the Tits fibration
$X \dasharrow Y$ which on the open orbit is the $G_0$-equivariant quotient map $U = G_0/H \to G_0/N(H_0)$
with respect to the natural actions of $G_0$ on $G_0/H$ and $G_0/N(H_0)$
(cf.~\cite[Propositions 1 and 6, page 61 and 65 ]{HO}, or \cite[\S 1.3]{Br07}).

If the Tits fibration is trivial, i.e., its image is a point,
then $G_0=N(H_0)$. Hence $H_0$ is a normal subgroup of $G_0$ and the quotient $G_0/H_0$ is
a connected linear algebraic group. Thus $-K_X$ is big by Lemma \ref{affine}.

Now assume that the Tits fibration $X \dasharrow Y$ is non-trivial, i.e.,
$\dim G_0/N(H_0) > 0$.
Taking $G_0$-equivariant blowups $\pi: X' \to X$ and $Y' \to Y$, we may assume that
$\pi^*(-K_X) = L + E$ and a base point free linear system $\Lambda \subseteq |L|$
gives rise to the Tits fibration $f: X' \to Y'$ (cf.~\cite{HO}, \cite{Br07}).
Write $L = f^*A$ with $A$ very ample.
Write $K_{X'} = \pi^*K_X + E'$ with $E'$ effective.
Let $F$ be a general fibre of $f$. Then $F$ is almost homogeneous under the
action of the linear algebraic group $N(H_0)/H_0$ and the latter is isomorphic
to the open orbit of $F$.
Hence $-K_F$ is big by Lemma \ref{affine}.

So $-\pi^*K_X | F = -K_{X'} | F + E'| F = - K_F + E'| F$ is big
i.e., $-\pi^*K_X$ is relatively big over $Y'$, which is also effective by the discussion above.
By \cite[Lemma 2.5]{CCP}, the divisor
$-\pi^*K_X + f^*A = L + E + f^*A = 2L + E$ is big.
Thus
$\kappa(X', -\pi^*K_X) = \kappa(X', L + E) = \kappa(X', 2L + E) = \dim X'$.
So $-\pi^*K_X$ and hence $-K_X$ are both big.
This proves Theorem \ref{ThB}(1).

\subsection{Proof of Theorem \ref{ThA}}

Since $(X, G)$ is strongly primitive, there is no non-trivial $G$-equivariant fibration.
In particular, the Kodaira dimension $\kappa(X) \le 0$,
noting that $|G| = \infty$ and that a variety $Y$ of general type is known to have
finite $\Aut(Y)$.
Let
$$\bar{G}_0 \le \Aut_0(X)$$
be the Zariski-closure of $G_0$ which is normalized by $G$ and of dimension $\geq 1$
by our assumption on $G_0$, and let $\bar{G}_{00} := (\bar{G}_0)_{0}$ be its connected component.

If $\bar{G}_0$ does not have a Zariski-dense open orbit in $X$, then
as in \cite[Lemma 2.14]{Z-Tits} with $H := \bar{G}_{00}$
which is normalized by a finite-index subgroup $G_1$ of $G$,
there is a $G_1$-equivariant non-trivial fibration,
contradicting the strong primitivity of $(X, G)$.
So we may assume that $X$ is almost homogeneous under the action of the algebraic
group $\bar{G}_0$.

Suppose that $q(X) = 0$.
Then $\Aut_0(X)$ is a linear algebraic group (cf.~\cite[Theorem 3.12]{Li}).
The natural composition
$G \to G . \bar{G}_0 \to (G . \bar{G}_0)/((G . \bar{G}_0) \cap \Aut_0(X))$ induces
the first injective homomorphism below while the middle one is due to
the second group isomorphism theorem:
$$G/G_0 \hookrightarrow (G . \bar{G}_0)/((G . \bar{G}_0) \cap \Aut_0(X))
\cong ((G . \bar{G}_0) . \Aut_0(X))/\Aut_0(X) \le \Aut(X)/\Aut_0(X)$$
where the last group is finite by Theorem \ref{ThB}.
This contradicts the assumption.

Suppose now that $q(X) > 0$.
Let $\alb_X : X \to A := \Alb(X)$ be the Albanese map which is automatically $\Aut(X)$- and hence
$G$-equivariant,
and which must be generically finite onto the image $\alb_X(X)$ by the strong primitivity of
$(X, G)$. Hence $\kappa(X) \ge \kappa(\alb_X(X)) \ge 0$.
Thus $\kappa(X) = 0$. So $\alb_X$ is a bimeromorphic and surjective morphism (cf.~\cite[Theorem 24]{Ka}).

Since $X$ is almost homogeneous under the action of $\bar{G}_0$ and also
of $\bar{G}_{00}$, so is $A$ under
the action of $\bar{G}_{00 |A}$. Hence $\bar{G}_{00 |A} = \Aut_0(A) = A$.
We still need to show that $\alb_X : X \to A$ is an isomorphism.
Suppose the contrary that we have a non-empty exceptional locus $E \subset X$ over which
$\alb_X$ is not an isomorphism.
Then both $E$ and $F := \alb_X(E)$ are stable under the actions of $\bar{G}_{00}$, and
$\bar{G}_{00 |A} = A$, respectively.
Hence $\dim F \geq \dim A$, contradicting the fact that $\alb_X$ is a bimeromorphic map.
Theorem \ref{ThA} is proved.

\section{Proof of Applications}
\subsection{ Proof of Application \ref{Cor1}}

Set $G := \Aut(X)$ and $G_0 := \Aut_0(X)$.
Since $q(X) = 0$, $G_0$ is a linear algebraic group (cf.~\cite[Theorem 3.12]{Li}).
By \cite[Lemma 4.2]{Fu} and since $G_0 \lhd G$,
there is a quotient map $X \dasharrow Y = X/G_0$
such that the action of $G$ on $X$ descends to a (not necessarily faithful) action of $G$ on $Y$.
Taking a $G$-equivariant resolution $Y' \to Y$ and
a $G$-equivariant resolution $X' \to \Gamma_{X/Y'}$ of the graph of
the composition $X \dasharrow Y \dasharrow Y'$,
the natural map $f: X' \to Y'$ is holomorphic and $G$-equivariant.
A general fibre $F$ of $f$
is almost homogeneous under the action of $G_0$.
If $\dim Y' = 0$, then Application \ref{Cor1} follows from Theorem \ref{ThB}.

Suppose that $\dim Y' = 1$. Since $q(Y') \le q(X') = q(X) = 0$, $Y' \cong \PP^1$.
So $-K_{Y'}$ is ample.
By Theorem \ref{ThB}, $-K_{X'} | F = -K_F$ is big, i.e., $-K_{X'}$ is relatively big over $Y'$.
So $B' := -K_{X'} + m f^*(-K_{Y'})$, with $m >> 1$, is a big divisor
(cf.~Proof of \cite[Lemma 2.5]{CCP}), whose class is stabilized by $G$.
Now $B = \pi_* B'$, with $\pi : X' \to X$ the natural birational morphism,
is a big divisor on $X$ whose class is stabilized by $G$.
Thus Application \ref{Cor1} follows from Lemma \ref{big}.

\subsection{}\label{setup1.1}
We recall some basic notions from dynamics.
Let $X$ be a compact K\"ahler manifold.
For an automorphism $g \in \Aut(X)$, its (topological)
{\it entropy} $h(g) = \log \rho(g)$ is defined as
the logarithm of the {\it spectral radius} $\rho(g)$, where
$$\rho(g) := \max \{|\lambda| \, ; \, \lambda \,\,\,
\text{is an eigenvalue of} \,\,\,
g^* | \oplus_{i \ge 0} H^i(X, \C)\}.$$
By the fundamental result of Gromov and Yomdin,
the above definition is equivalent to the original dynamical definition of entropy
(cf.\ \cite{Gr}, \cite{Yo}).

An element $g \in \Aut(X)$ is of {\it null entropy}
if its (topological) entropy $h(g)$ equals $0$.
For a subgroup $G$ of $\Aut(X)$, we define the {\it null subset} of $G$ as
$$N(G) := \{g \in G \, | \, g \, \,
\text{\rm is of null entropy, i.e.,} \, h(g) = 0\} $$
which may {\it not} be a subgroup.
A group $G \le \Aut(X)$ is of {\it null entropy} if every $g \in G$
is of null entropy, i.e., if $G$ equals $N(G)$.

For $g \in \Aut(X)$, let
$$d_1(g) := \max\{|\lambda| \, ; \, \lambda \,\, \text{is
an eigenvalue of} \,\, g^* \, | \, H^{1,1}(X) \}$$ be the {\it first dynamical degree} of $g$
(cf.~\cite[\S 2.2]{DS}), which is $\geq 1$.

Let $X$ and $Y$ be compact K\"ahler manifolds of dimension $n$ and $l$, with K\"ahler forms  (or ample divisors when $X$ and $Y$ are projective)
$\omega_X$ and $\omega_Y$, respectively.
For a surjective holomorphic map $\pi: X \to Y$ and an automorphism $g \in \Aut(X)$, the first relative dynamical degree of $g$ (cf.~\cite[\S 3]{DNT}) is defined as
$d_1(g|\pi) = \lim_{s \to \infty} \lambda_1(g^s|\pi)^{1/s}$; here
$$\lambda_1(g^s | \pi) = \langle (g^s)^* \omega_X \wedge \pi^* \omega_Y^l, \omega_X^{n-1-l} \rangle;$$
it depends on the choice of K\"ahler forms, but $d_1(g|\pi)$ does not, because of sandwich-type
inequalities between positive multiples of any two K\"ahler forms.

\subsection{Proof of Application \ref{ThC}}

By the Khovanskii-Tessier inequality,
the first dynamical degreee $d_1(g) = 1$ if and only if the topological entropy $h(g) = 0$
(cf.~\cite[Corollaire 2.2]{DS}), hence the second claim follows from the first one.
By \cite[Theorem 1.1]{DNT}, $d_1(g_{| X}) = \text{max} \{ d_1(g_{| Y}), d_1(g| \pi)\}$. Thus, to prove Application \ref{ThC}, it suffices to show $d_1(g | \pi) \le 1$.

Let $F$ be a general fibre of $\pi : X \to Y$. Then $F$ is almost homogeneous
under the action of  $G_0$. Hence $-K_X | F = -K_F$
is big by Theorem \ref{ThB}, i.e., $-K_X$ is relatively big over $Y$.
Let $A$ be an ample divisor on $Y$ such that $-K_X + \pi^*A$ is a big divisor on $X$
(cf.~Proof of \cite[Lemma 2.5]{CCP}) and is hence equal to $L + E$
for an ample $\Q$-divisor $L$ and an effective $\Q$-divisor $E$.

Set $n := \dim X$ and $\ell := \dim Y$.
Noting that $\pi^*A \cdot F = 0$,
we have
$$\begin{aligned}
\lambda_1(g^s | \pi) :&= (g^s)^* (L) \cdot L^{n-1-\ell} \cdot F
\le (g^s)^* (L+E) \cdot L^{n-1-\ell} \cdot F \\
&= (g^s)^* (-K_X) \cdot L^{n-1-\ell} \cdot F
= (-K_X) \cdot L^{n-1-\ell} \cdot F = :c
\end{aligned}$$
where the last term is a positive number independent of $s$.
Hence
$$d_1(g | \pi) = \dim_{s \to \infty} \, (\lambda_1(g^s | \pi))^{1/s} \le \dim_{s \to \infty} \, c^{1/s} = 1.$$
This proves Application \ref{ThC}.

\subsection{Remarks} \label{Rmk}

(i) If we denote by $K$ the kernel of the natural surjective homomorphism $G_{| X} \to G_{| Y}$, then we have an equality of
null sets (as sets of cosets):
$$
N(G_{| X})/K = N(G_{| Y}).
$$

(ii) If the irregularity $q(X) = 0$, then $\Aut_0(X)$ and hence $G_0$ are always linear algebraic groups
(cf.~\cite[Theorem 3.12]{Li}). In this case,
Application \ref{ThC} implies that we may
assume that $\Aut_0(X) = (1)$ when studying dynamics.

(iii) Application \ref{ThC} does not hold if
 $G_0$  is not linear. For example, let $X = T_1 \times T_2$ be the product of two complex tori and $G_0 = \Aut_0(T_1)$. The quotient space $Y:=X/G_0$ is $T_2$. Suppose that $g \in \Aut(T_1)$ is an element of positive entropy
 which acts trivially on $T_2$. Then we have
 $d_1(g_{| X}) > d_1(g_{| Y}) = 1$.


\begin{thebibliography}{99}

%
%
%
%

\bibitem{Br07}
M.~Brion,
Log homogeneous varieties,
Proceedings of the XVIth Latin American Algebra Colloquium, 1�-39,
Bibl. \ Rev. \ Mat. \ Iberoamericana, Madrid, 2007,
arXiv:math/\textbf{0609669}



\bibitem{CWZ}
F.~Campana, F.~Wang, D.~-Q.~Zhang, Automorphism groups of positive entropy on prjective threefolds,
Trans. \ Amer. \ Math. \ Soc. \ (to appear), arXiv:\textbf{1203.5665}

\bibitem{CCP}
F.~Campana, J. A.~Chen, T.~Peternell,
Strictly Nef Divisors, Math.\ Ann.\  \textbf{342} (2008), no.~3, 565-�C585.


\bibitem{Ca2}
S.~Cantat, Endomorphismes des vari\'etes homog\`enes,  Enseign.\ Math. \ \textbf{49} (2004), 237--262.

\bibitem{DNT}
T.~-C.~Dinh, V.~-A.~Nguyen and T.~T.~Truong,
On the dynamical degrees of meromorphic maps preserving a fibration,
arXiv:\textbf{1108.4792}

\bibitem{DS}
T.-C.~Dinh and N.~Sibony,
Groupes commutatifs d'automorphismes d'une vari\'et\'e k\"ahlerienne compacte,
Duke Math.\ J.\ \textbf{123} (2004), no.~2, 311--328.

\bibitem{Fu}
A.~Fujiki,
On automorphism groups of compact K\"ahler manifolds,
Invent.\ Math.\ \textbf{44} (1978), no.~3, 225--258.

\bibitem{Gr}
M.~Gromov, On the entropy of holomorphic maps,
\ Enseign. \ Math. (2)  \textbf{49} (2003),  no. 3--4, 217--235.

\bibitem{Hb}
B.~Harbourne,
Rational surfaces with infinite automorphism group and no antipluricanonical curve,
Proc. \ Amer. \ Math. \ Soc. \textbf{99} (1987), no. 3, 409--414.

\bibitem{Hs}
R.~Hartshorne,
Ample Subvarieties of Algebraic Varieties,
Lecture Notes in Mathematics, Vol. \textbf{156} (1970).

\bibitem{HT}
B.~Hassett and Y.~Tschinkel,
Geometry of equivariant compactifications of $\GG_a^n$,
Internat.\ Math. \ Res.\ Notices \textbf{22} (1999), 1211-1230.

\bibitem{HO}
A.~Huckleberry, E.~Oeljeklaus, Classification theorems for almost homogeneous spaces,
Institut \'Elie Cartan, \textbf{9}. Universit\'e de Nancy, 1984.

\bibitem{Ka} Y.~Kawamata,
Characterization of abelian varieties,
Compos.\ Math.\ \textbf{43} (1981), 253--276.

\bibitem{Mo} B.~G.~Moishezon,
On $n$-dimensional compact varieties with $n$ algebraically independent meromorphic functions, Amer. Math. Soc. Translations \ \textbf{63} (1967), 51-177


\bibitem{Li}
D.~I.~Lieberman,
Compactness of the Chow scheme: applications to automorphisms
and deformations of K\"ahler manifolds,
pp.~140--186,
Lecture Notes in Math.\ \textbf{670}, Springer, 1978.


\bibitem{Yo}
Y.~Yomdin,
Volume growth and entropy,
Israel J. \ Math. 57 (1987), no. 3, 285--300.

\bibitem{JDG}
D.~-Q.~Zhang, Dynamics of automorphisms on projective complex manifolds,
J.\ Differential Geom.\ \textbf{82} (2009), no.~3, 691--722.

\bibitem{Z-Tits}
D.~-Q.~Zhang,
A theorem of Tits type for compact K\"ahler manifolds,
Invent.\ Math.\ \textbf{176} (2009), no.~3, 449--459.

\end{thebibliography}
\end{document}